\documentclass[a4paper,UKenglish,cleveref, autoref, thm-restate]{lipics-v2021}
\nolinenumbers
\title{Cosine and Computation} %TODO Please add

\author{Prabhat Kumar Jha}{Tata Institute of Fundamental Research, Mumbai, India}{000prabhat000@gmail.com}{https://orcid.org/0000-0001-6225-9147}{}

\authorrunning{P. K. Jha} %TODO mandatory. First: Use abbreviated first/middle names. Second (only in severe cases): Use first author plus 'et al.'

\Copyright{Prabhat Kumar Jha} %TODO mandatory, please use full first names. LIPIcs license is "CC-BY";  http://creativecommons.org/licenses/by/3.0/

\ccsdesc[500]{Theory of computation~Logic and verification}
\ccsdesc[500]{Theory of computation~Verification by model checking}
\ccsdesc[100]{Mathematics of computing~Ordinary differential equations}
\keywords{Matrix, Orbit, Subspace, Reachability, Verification, Recurrence, Linear, Continuization, Cosine} %TODO mandatory; please add comma-separated list of keywords

\category{} %optional, e.g. invited paper

\relatedversion{} %optional, e.g. full version hosted on arXiv, HAL, or other respository/website
\funding{We acknowledge support of the Department of Atomic Energy, Government  of India, under project no. RTI4001}%optional, to capture a funding statement, which applies to all authors. Please enter author specific funding statements as fifth argument of the \author macro.

\acknowledgements{I want to thank my MSc thesis advisors S. Akshay and Piyush Srivastava for introducing me to Skolem problem and encouraging me to write this paper.}%optional

\EventEditors{}
\EventNoEds{2}
\EventLongTitle{}
\EventShortTitle{}
\EventAcronym{}
\EventYear{}
\EventDate{}
\EventLocation{}
\EventLogo{}
\SeriesVolume{}
\ArticleNo{}
\begin{document}

\maketitle

%TODO mandatory: add short abstract of the document
\begin{abstract}
We are interested in solving decision problem $\exists? t \in \mathbb{N}, \cos t \theta = c$ where $\cos \theta$ and $c$ are algebraic numbers. 
We call this the $\cos t \theta$ problem.
This is an exploration of Diophantine equations with analytic functions.
Polynomial, exponential with real base and cosine function are closely related to this decision problem: $ \exists ? t \in \mathbb{N}, u^T M^t v = 0$ where $u, v \in \mathbb{Q}^n, M \in \mathbb{Q}^{n\times n}$.
This problem is also known as ``Skolem problem'' and is useful in verification of linear systems. 
Its decidability remains unknown.
Single variable Diophantine equations with exponential function with real algebraic base and $\cos t \theta$ function with $\theta$ a rational multiple of $\pi$ is decidable.
This idea is central in proving the decidability of Skolem problem when the eigenvalues of $M$ are roots of real numbers.
The main difficulty with the cases when eigenvalues are not roots of reals is that even for small order cases decidability requires application of trancendental number theory which does not scale for higher order cases.
We provide a first attempt to overcome that by providing a $PTIME$ algorithm for $\cos t \theta$ when $\theta$ is not a rational multiple of $\pi$.
We do so without using techniques from transcendental number theory.
\par
One of the main difficulty in Diophantine equations is being unable to use tools from calculus to solve this equation as the domain of variable is $\mathbb{N}$.
We also provide an attempt to overcome that by providing reduction of Skolem problem to solving a one variable equation (which involves polynomials, exponentials with real bases and $\cos t \theta$ function with $t$ ranging over reals and $\theta \in [0, \pi]$) over reals. 

\end{abstract}
\section{Introduction}
\label{sec:intro}%TODO: Add Introduction to the file below
%! Author = prabhat
%! Date = 11/07/21

Reachability problems are special type of verification problems which ask if a given system can ever reach a given configuration.
If behaviour of the system is deterministic and can be described as a function of time then the problem reduces to solve some equation. 
One such example is discrete-time linear dynamical systems of which the behaviour can be described as a linear recurring sequence. 
The problem of checking existence of a $0$ in a given linear recurring sequence (LRS) is known as the Skolem problem and decidability of this problem remains unknown.
Finding existence of $0$ in such sequences over $\mathbb{Q}$ reduces to finding existence of solution of exponential Diophantine equations over algebraic numbers. 
The Skolem problem is in $NP^{RP}$ when eigenvalues of the given LRS are roots of real numbers\cite{Akshay2017ComplexityOR}. 
Hence, the challenge is to solve the cases when eigenvalues are not roots of reals and only known decidability results are for order 2 and 3 using Baker's method of linear forms of logarithms\cite{Shorey}. 
Here we will study a basic case of the exponential Diophantine equations which we will call `` $\cos t \theta$ problem''. 
\par 
Given  real algebraic numbers $\cos \theta$ and $c$ between $-1$ to $1$, the $\cos t \theta$ problem asks if there is a natural number $t$ such that $\cos t \theta = c$. 
This is a special case of Skolem problem of order 3 over algebraic numbers and hence is known to be decidable but the bounds obtained by Baker's method is exponential which yields $NP^{RP}$ complexity.
\par
Given a square matrix $M$, a vector $u$ and an affine subspace $W$, the affine subspace reachability problem asks if there is a natural number $t$ such that $M^t u \in W$. 
Orbit problem is $0$-dimensional case of affine subspace reachability problem and is known to be decidable in $P$\cite{KannanPtime, Kannan_decidable}.
\par
Our first contribution is a polynomial time algorithm for the $\cos t \theta$ problem by reducing this problem to Orbit problem over $\mathbb{Q}$.
Our reduction uses some facts from algebraic number theory and the algorithm for Orbit problem also uses algebraic number theory and some properties of matrices. 
Hence we do not need the transcendental number theory. 
%Given a square matrix $M$, a vector $u$ and an affine subspace $W$, the affine subspace reachability problem asks if there is a natural number $t$ such that $M^t u \in W$. 
%Orbit problem is $0$-dimensional case of affine subspace reachability problem. 
Skolem problem is equivalent to affine subspace reachability problem under polynomial reductions. 
Our method is first application of Orbit problem to solve a non-trivial case of Skolem problem.
\par 
We then consider two generalizations of $\cos t \theta$ problem.
The first one is $r^t \cos t \theta$ problem and second one is $\Sigma \cos t \theta_i$ problem. 
These problems are special cases of Skolem problem over algebraic numbers. 
The first one is known to be decidable using Baker's method as it is of order 3 but finding an efficient algorithm or a better lower bound remains unknown. 
Another possible way to solve that is using affine subspace reachability problem of dimension 1 \cite{Orbit_higher} but that also requires Baker's method and the gap between known upperbound and lowerbound remains unchanged.
The second one is a not known to be decidable. 
Only a special case when $\theta_i$s are rational multiple of $\pi$ is known to be $NP$-complete.\cite{Akshay2017ComplexityOR} 
\par
Our second contribution is to give a polynomial time  reduction of $\exists? t \in \mathbb{N}, \Sigma \cos t \theta_i = c $ to $\exists? t \in \mathbb{R}, \Sigma \cos t \theta_i = c $. 
%This technique shows that solving exponential Diophantine equations over algebraic numbers is polynomial time reducible to exponential equations over $\mathbb{C}$. 
The same technique gives us that the Skolem problem can be reduced to $\exists? t \in \mathbb{R}, \Sigma r_i^t p_i(t) \cos t \theta_i = c $ where $r_i$ is an algebraic number and $p_i$ is a polynomial.
This problem is a special case of one variable restriction of extension of theory of real numbers with cosine and power function. 
Unrestricted case is known to be undecidable as solving the Diophantine equation with 4 or more variables is undecidable. 
\par
We will go through some preliminaries of computation with algebraic numbers in 2 \ref{sec:prelim}.
Specifically we will see how to represent algebraic numbers and the complexity of basic operations.
We will also go through basics of linear recurring sequences and the orbit problem.
The algorithm and its analysis for the $\cos t \theta$ problem is presented in section 3 \ref{sec:cos}. 
In section 4 \ref{sec:ext}, the extensions of the $\cos t \theta$ problem has been studied.
Second contribution has been provided in section 5 \ref{sec:conti}. 
Finally, we give a conclusion and open problems in section 6 \ref{sec:con}.

\section{Preliminaries}
\label{sec:prelim}
%! Author = prabhat
%! Date = 11/07/21

% Preamble
In this section, we will go through some preliminaries of computations with algebraic numbers and linear recurring sequences.
We begin with introduction to computation with algebraic numbers. 
We refer to Cohen \cite{cohenbook} for this topic.
\subsection{Computation with Algebraic Numbers}
\label{algebraic:prelim}
In the study of matrices over $\mathbb{Q}$, the eigenvalues are from a subfield of $\mathbb{C}$ and that is exactly what is known as algebraic numbers. 
Since discrete-time linear systems are represented using matrices and eigenvalues are very important to study properties of matrices, we are interested in algebraic numbers.
We begin with defining algebraic numbers:
\begin{definition}[Algebraic Numbers]
\label{def: algebraic}
 A complex number $\alpha$ is said to be an algebraic number if there is a polynomial $p \in \mathbb{Z}[x]$ such that $p(\alpha) = 0$.
 There is a unique polynomial with minimal degree with greatest common divisor of coefficients 1 and it is said to be the minimal polynomial of $\alpha$.
 Degree $D(\alpha)$ is degree of the minimal polynomial of $\alpha$.
 Height $H(\alpha)$ is maximum absolute value of a coefficient in minimal polynomial of $\alpha$.
 Roots of minimal polynomial of $\alpha$ are called Galois conjugates of $\alpha$.
 Norm $\mathcal{N}(\alpha)$ is product of Galois conjugates of $\alpha$.
 If the leading coefficient of minimal polynomial is 1 then $\alpha$ is said to be an algebraic integer.
 $\mathbb{A}$ denotes set of all algebraic numbers and $\mathcal{O}_\mathbb{A}$ denotes set of all algebraic integers.
\end{definition}

\par 
Now that we have defined algebraic numbers, for the purpose of computation, we need to represent them. 
We represent integers as a binary string and rational numbers as a pair of integers. 
The canonical representation of an algebraic number is defined as below.

\begin{definition}
 \label{def: representation}
 The canonical representation of an algebraic number $\alpha$ is a tuple $(P, x, y, r)$ where $P$ is the minimal polynomial of $\alpha$ and $x, y, r \in \mathbb{Q}$ such that $\alpha$ is in the circle centered at $x +\iota y$ with radius $r$.
 $x, y, r$ are choosen to distinguish $\alpha$ from its Galois conjugates.
\end{definition}
Note that the canonical representation is not unique but still it is trivial to check the equality of two algebraic numbers.
We also need to make sure that $x, y, r$ do not have a very large representation as that can increase the complexity.
The theorem below gurantees that.
\begin{theorem}~\cite{Mignotte1983}
 \label{thm:Mignotte_root}
If two conjugates $\alpha_i, \alpha_j$  of $\alpha$ are not equal then $|\alpha_i - \alpha_j| > \frac{\sqrt{6}}{d^{\frac{(d+1)}{2}}H^{d-1}}$.
\end{theorem}
While canonical representation is common in literatute, we will need another representation in order to use algebraic numbers as matrices. 
This uses some applications of LLL algorithm and we will directly use the following results without proving. 
One can have a look in section 2.6 of Cohen ~\cite{cohenbook} for details.
\begin{theorem}
 \label{thm:lattice_reduction}
There is a polynomial time algorithm which takes $z_1, z_2,...,z_k, z \in \mathbb{A}$ as input and outputs whether $z$ is a $\mathbb{Q}$-linear combination of $z_1, z_2, ..., z_k$.
In case of positive answer it also outputs the coefficients.
\end{theorem}

Theorem \ref{thm:lattice_reduction} provides a way to write an algebraic number as a $\mathbb{Q}$-vector in a suitable number field.
Using this and elementary linear algebra, one can write the matrix of multiplication with an algebraic number in polynomial time.
In this vector representation, doing basic operations such as addition, multiplication and division are all computable in polynomial time.
This fact is important to our main result in section \ref{sec:cos}.
\subsection{From Skolem problem to Diophantine equations}
\label{Skolem:prelim}
In this section, we will go through the basics of linear recurring sequences in order to understand the Skolem problem.
\begin{definition}[Linear Reccuring Sequences]
 \label{def:LRS}
 A linear recurring sequence (LRS) of order $k$ over ring $R$ is a sequence of elements of $R$ which satisfies  $\forall t> k, a_t = \Sigma_{1 \leq i \leq k} c_i a_{t-i}$ where $c_i \in R$. 
\end{definition}
An LRS of order $k$ can be determined from first $k$ terms as rest of the terms can be deterministically computed using the recurrence relation.
We are interested in cases when $R$ is one of $ \mathbb{Z}, \mathbb{Q}, \mathbb{A}$. 
There is an interesting result about LRS over fields of characteristic 0.
This theorem is about the 0s of an LRS.
\begin{theorem}[Skolem-Mahler-Lech]~\cite{lech1953}
 \label{thm:SML}
 The zeros of LRS over a field of characteristic $0$ is a union of a finite set and finitely many arithmetic progressions.
\end{theorem}
The known proof of Theorem \ref{thm:SML} uses $p$-adic methods and proof by contradiction.
The proof is of non-constructive nature. 
That is this proof does not provide an algorithm to check whether a given LRS has a 0.
The problem of finding 0 is known as the Skolem problem.
Following folklore claim gives another definition of LRS in terms of matrices.
\begin{claim}
 \label{clm:MatrixLRS}
 Given a $k \times k$ matrix $M$ and $k$-dimensional vectors $u$ and $v$, the sequence $a_t = u^T M^t v$ is an LRS.\\
 Given any LRS $\{a\}_t$, there exist a matrix $M$ and vectors $u$ and $v$ such that $a_{k+t} =  u^T M^t v$ where $k$ is the order of LRS..
\end{claim}
\begin{proof}
 Let's consider the characteristic polynomial of $M$, let it be $x^d - \Sigma_{i = 1}^{i = d} a_i x^{d-i}$.
 Caley-Hamilton theorem implies that $M^d = \Sigma_{i=1}^{i=d} M^{d-i}$.
 Multiplying $M^{t-d}$ both sides, we get that $M^t = \Sigma_{i=1}^{i=d}M^{t-i}$.
 By multiplying the vectors $u$ and $v$ and using linearity we get, $u^TM^tv = \Sigma_{i=1}^{i=d}u^TM^{t-i}v$.
 From Definition \ref{def:LRS}, it follows that $u^TM^tv$ is an LRS. 
 \par
 Let the first $k$ terms be $a_1,...,a_k$ and the recurrence be $a_t = \Sigma_{i = 1}^{k}c_i a_{t-k}$. 
 Let $M$ be:
 $$\begin{bmatrix}
  c_1 & c_2 & ... & c_{k-1} & c_k \\
    & \mathbf{I}_{k-1 }   &  &  & \mathbf{0}
 \end{bmatrix} $$
 where $\mathbf{I}_{k-1}$ is identity matrix of order $k-1$ and $\mathbf{0}$ is a column matrix of size $k-1$ with $0$ as all of its entries.
 Let $u = (1, 0, ..., 0)^T$ and $v = (a_k, a_{k-1}, ..., a_1)^T$. Now using induction we can verify that $u^T M^t v = a_{k+t}$.

\end{proof}

Now we have another version of Skolem problem that is checking if $u^T M^t v$ is $0$ for some $t$.
For the case when LRS is over rational numbers or over algebraic numbers, the equation $u^T M^t v = 0$ has a closed form. 
It can be obtained using Jordan canonical form and properties of matrix multiplication. 
If the eigenvalues are $\lambda_1,..., \lambda_m$ then the closed form equation is of the form $\Sigma_{i = 1}^{i = m} p_i(t) \lambda_i^t = 0 $, where $p_i \in \mathbb{Z}[x]$.
In the case when LRS is over rational numbers we know that eigenvalues occur as conjugates and are algebraic so we get the equation $\Sigma_{i = 1}^{i = m} p_i(t) r_i^t \cos t \theta_i$ where $p_i \in \mathbb{Z}[x], r_i, \cos \theta_i \in \mathbb{R} \cap \mathbb{A}$ and $|\cos\theta_i|\leq 1$.
We state this as the following lemma:
\begin{lemma}
 \label{lem:exp_eqn}
 Skolem problem over rational numbers (or integers) can be reduced to solving equation $\Sigma_{i = 1}^{i = m} p_i(t) r_i^t \cos t \theta_i$ where $p_i \in \mathbb{Z}[x], r_i, \cos \theta_i \in \mathbb{R} \cap \mathbb{A}$ and $\theta_i \in [0, \pi]$.
\end{lemma}
%This lemma immediately implies some of the well known properties of LRS such as closure under addition.
We will use this form of Skolem problem throughout this paper.
\subsection{Affine Subspace Reachability Problem}
\label{ASRP:main}
Affine Subspace Reachability Problem asks if a given  linear system reaches to a given affine subspace after some steps. 
We define this problem precisely here.
\begin{definition}

 \label{def:ASRP}
 Input: $M \in \mathbb{Q}^{k \times k}, v \in \mathbb{Q}^k$ and an affine subspace $W$ described using linear equations it satisfies.\\
 Output: ``Yes'' if there is a $t \in \mathbb{N}$ such that $M^t v \in W$; ``No'' otherwise.
\end{definition}

Affine subspaces are defined using equations of form $u^T v = c$. 
It is trivial that Skolem problem is a special case of affine subspace reachability problem. 
An interesting and folklore converse is that affine subspace reachability problem is polynomial time reducible to Skolem problem.
We will now see a reduction to Skolem problem.
\begin{theorem}[Reduction to Skolem Problem (Folklore)]
 \label{thm: redSkolem}
 Affine subspace reachability problem is polynomial time reducible to Skolem problem.
\end{theorem}
\begin{proof}[Proof-Sketch]
 If Skolem problem is decidable then we can also compute the zero set explicitly. 
 Affine subspace reachability problem is like finding intersection of solution sets of equations of type $u^T M^t v = 0$.
 Intersection of arithmetic progressions can be computed using chinese remainder theorem.
\end{proof}
Orbit problem is 0-dimensional affine subspace reachability problem.
\begin{definition}[Orbit problem]
 \label{def:orbit}
 Input: $M \in \mathbb{Q}^{k \times k}, u,v \in \mathbb{Q}^k$
 Output: ``Yes'' if there is a $t \in \mathbb{N}$ such that $M^t u = v$; ``No'' otherwise.
\end{definition}
This problem is known to be decidable in PTIME.
The techniques used are from algebraic number theory.
The link between Skolem problem and Orbit problem was also hinted in ~\cite{KannanPtime}.
\begin{theorem}[Complexity of Orbit problem]~\cite{KannanPtime, Kannan_decidable}
\label{thm:compOrbit}
 The Orbit problem is in P.
\end{theorem}

\section{The $\cos t \theta$ Problem}
\label{sec:cos}
%%%%! Author = prabhat
%! Date = 11/07/21

% Preamble
In this section we will provide a polynomial time reduction from the $\cos t \theta$ problem to the Orbit problem. 
We first prove that the sequence $a_t = \cos t \theta$ is an LRS.
\begin{theorem}[$\cos t \theta$ is an LRS]
 \label{thm:coslrs}
 The sequence $a_t = \cos t \theta$ satisfies a linear recurrence relation over $\mathbb{Q}$.
\end{theorem}

\begin{proof}
 Let $z = \cos \theta + \iota \sin \theta$ where $\sin \theta = \sqrt{1 - \cos^2 \theta}$.
 Consider the minimal polynomial of $z$, $c_0x^k - \Sigma_{i = 1}^{i = k}c_i x^{k-i}$. 
 That implies, $z^k = \Sigma_{i = 1}^{i = k}z^{k-i}$.
 Multiplying with $z^{t-k}$ both sides we get, $z^t = \Sigma_{i = 1}^{i = k}z^{t-i}$.
 Taking real parts of both sides and using De'Moivere's identity, $\cos t \theta = \Sigma_{i = 1}^{i = k} \cos (t-i) \theta$.
 We get a linear recurrence relation.
\end{proof}

The eigenvalues of this LRS are exactly the conjugates of $z$.
The equation $\cos t \theta =c$ is an affine subspace reachability problem of co-dimension 1.
This can also be thought of as Skolem problem of order 3 over algebraic numbers.

Now we will see a reduction from this problem to Orbit problem.
\subsection{$\cos t \theta$ is in PTIME}
\label{cos:main}

The reduction exploits the fact that multiplication with algebraic numbers is a $\mathbb{Q}$-linear transformation. 
\begin{theorem}[$\cos t \theta$ problem is in P]
 \label{thm:main}
 Given real algebraic numbers $\alpha = \cos \theta, c$ such that $|\alpha| \leq 1, |c| < 1$, there is a polynomial time algorithm to check the existence of a natural number $t$ such that $\cos t \theta = c$.
\end{theorem}

We provide the following algorithm.
\begin{enumerate}
 \item Compute $z  = \alpha + \iota \sqrt{1 - \alpha^2} $
 \item Check if $c \pm \iota \sqrt{1-c^2} \in \mathbb{Q}(z)$. If both cases give negative answer return ``No'', otherwise compute the coordinate of the target vectors (those amongst $c \pm \iota \sqrt{1-c^2}$ which are in $\mathbb{Q}(z)$)  and go to next step.
 \item Compute the  multiplication matrix for $z$.
 \item Solve $M^t \mathbf{1} = v$ for all target vectors.
 \item Return $OR$ of outputs.
\end{enumerate}
Below we provide a proof of Theorem \ref{thm:main}. 

\begin{proof}
As $z = \cos \theta + \iota \sin \theta$ using De'Moivere's identity, $z^t = \cos t \theta + \iota \sin t \theta$.
If $\cos t \theta = c$ then $\sin t \theta$ is either $\sqrt{1-c^2}$ or $-\sqrt{1-c^2}$.
We can consider both the cases.
So we need to check $\exists ?t \in \mathbb{N} z^t = c + \iota \sqrt{1-c^2}$ or $\exists? t \in \mathbb{N} z^t = c - \iota \sqrt{1-c^2}$.
Step 2 checks if both of these are not in $\mathbb{Q}(z)$, since $z^t \in \mathbb{Q}$, so we only need to check for those targets which are in $\mathbb{Q}(z)$. 
This condition can be checked in polynomial time as mentioned in Theorem \ref{thm:lattice_reduction}.
Theorem \ref{thm:lattice_reduction} also gives coordinate for the case when it is in $\mathbb{Q}(z)$.
Multiplication with $z$ is a linear transformation over $\mathbb{Q}(z)$ which is a vector space over $\mathbb{Q}$.
We can compute this matrix in polynomial time as mentioned in Theorem \ref{thm:lattice_reduction}. 
Now the problem to check $\exists ? t \in \mathbb{N} z^t = c + \iota \sqrt{1-c^2}$ or $\exists t \in \mathbb{N} z^t = c - \iota \sqrt{1-c^2}$ is same as checking $\exists ? t \in \mathbb{N} M^t = v$ where $v$ is the vector representation for $c \pm \iota \sqrt{1-c^2}$.
This is an instance of Orbit problem.
The reduction is in polynomial time as all the required computation are done in polynomial time and number of steps is constant.
Using Theorem \ref{thm:compOrbit}, we get that $\cos t \theta$ problem is in P.

\end{proof}

\section{Extensions of $\cos t \theta$ Problem}
\label{sec:ext}
The $\cos t \theta $ problem is a natural problem from point of view of Diophantine equations with trigonometric functions.
%Our algorithm is sufficient for solving $f(t \theta) = c$ where $f$ is any trigonometric ratio.
This immediately suggests inquiry into extensions of the $\cos t \theta$ problem. 
We will look into two specific extensions.
The first one is due to exponential function while the second one is due to summation of LRSs.

We begin with the first extension.

\subsection{The $r^t \cos t \theta$ problem}
\label{rcos:ext}
Given an algebraic number $z$ and a real algebraic number $c$, checking the existence of $t$ such that $Re(z^t) = c$ is motivation for this extension.
This can also be written as $r^t \cos t \theta = c$ where $r = |z|$ and $\theta = arg(z)$.
We call this problem `` $r^t \cos t \theta$ problem''. 
Like $\cos t \theta$ problem $r^t \cos t \theta$ problem is also a case of Skolem problem of order 3 over algebraic numbers.
Hence this problem is also known to be decidable in $NP^{RP}$.
\begin{theorem}[Polynomial time restrictions of $r^t \cos t \theta$ problem]
For the following conditions the $r^t \cos t \theta$ problem is in P:\\
\begin{enumerate}
 \item $r \leq 1$
 \item $ z$ has a $\mathbb{Q}$-conjugate with absolute value less than or equal to 1
 \item $r = \frac{\alpha}{\beta} $  and $\cos \theta = \frac{\gamma}{\delta}$ where $\alpha, \beta, \gamma, \delta \in \mathcal{O}_{\mathbb{A}}$ such that ideal generated by $\alpha$ has a prime factor which does not divide ideal generated by $\delta$.
\end{enumerate}
\end{theorem}

\begin{proof}[Proof-sketch]
 1. The $\cos t \theta$ was a special case when $r = 1$. 
The case when $r < 1$ is also decidable in polynomial time as after $\left\lceil\frac{\log |c|}{\log |r|} \right \rceil$ steps the value of $r^t \cos t \theta < c$ at every later step.
\par
2. Using the Galois transformations $ z \mapsto \gamma$ we can convert $z^t + \overline{z}^t = c$ to $\gamma ^t + \overline{\gamma}^t = d$ where $\gamma$ is a conjugate with $|\gamma|\leq 1$. 
Then it is same as previous case.
\par
3. Using the valuation with respect to a prime factor of such an ideal, we get that the valuation will be monotonically increasing for $z^t$ and that gives a bound as the valuation of $c$ is fixed.
\end{proof}
The gap between upper and lower bounds remain as these cases are not exhaustive.

\subsection{The $\Sigma \cos t \theta_i$ problem}
\label{sumcos:ext}
Another way to extend this problem is by extending the order. 
This problem is $\exists t \in \mathbb{N}$ such that $\Sigma_{i = 1}^{i = k} c_i \cos t \theta_i = 0$.
We call this ``$\Sigma \cos t \theta_i$ problem.
This problem is not known to be decidable.
The $\cos t \theta$ is an special case of this problem.
This problem is known to be NP-hard\cite{Akshay2017ComplexityOR}.
Even a restriction of this problem when $\theta_i$s are restricted to be rational multiples of $\pi$ is known to be NP-complete.
Only case when we know decidability with non-degenerate $\theta$ is the $\cos t \theta$ problem.
We conjecture the following.
\begin{conjecture}
 $\Sigma \cos t \theta_i$ problem is decidable only if Skolem problem is decidable.
\end{conjecture}

\section{Contiuization and Computation}
\label{sec:conti}
In this section we will see few steps towards converting the Skolem problem to its analytical version. 
We will use continuization to do so. 
The motivation behind this is the fact that the equations $r^t p(t) \cos t \theta = c$ can be solved easily for $t \in \mathbb{R}$ where $\theta \in [0, \pi]$.
We state the following theorem as a first step towards continuization of Skolem problem.
\begin{proposition}
 If $\exists ?t \in \mathbb{R},  \Sigma \cos t \theta_i = 0$  is decidable then $\exists ?t \in \mathbb{N},  \Sigma \cos t \theta_i = 0$ is also decidable.
\end{proposition}
\begin{proof}
 We use the summation of squares method with the fact that $\cos 2 \pi t = 1$  characterises integers. 
$$\exists t \in \mathbb{N}, \Sigma c_i \cos t \theta_i = c \iff \exists t \in \mathbb{R}, ((\Sigma c_i \cos t \theta_i) - c)^2 + (\cos 2 \pi t - 1)^2  = 0 $$.

If we expand the square we get some multiplicative terms for example $\cos t \theta_i \cos t \theta_j$, using the identity $\cos(A+B) + \cos (A-B) = 2 \cos A \cos B$, we can convert them to additive cosine terms and get that the equation in rhs is also as desired  i.e. of form $\Sigma c_i \cos t \theta_i$ and we get the reduction.
\end{proof}

This can be extended to the Skolem problem also we omit the proof as it is very similar to previous one.

\begin{proposition}
 If $\exists ?t \in \mathbb{R},  \Sigma p_i(t)r_i^t \cos t \theta_i = 0$  is decidable then $\exists ?t \in \mathbb{N},  \Sigma p_i(t)r_i^t \cos t \theta_i = 0$ is also decidable.
\end{proposition}

Note that this different problem from continuous-time Skolem problem as the base is algebraic numbers for exponentials. 
However, this technique of continuization may be extended to membership problem for $P$-recursive sequences and other sequences.
We conjecture two statements one about $P$-recursive sequences and other about computation in general.
\begin{conjecture}
 The membership problem for $P$-recursive sequences is reducible to solving one variable equation over reals.
\end{conjecture}
\begin{conjecture}
\label{conj:computation}
 The one variable extension of real number with first order axiomatizable functions is decidable.
\end{conjecture}
Conjecture \ref{conj:computation} implies decidability of Skolem problem and  membership problem for $P$-recursive sequences.
If this conjecture is false then we get an extension of reals which is undecidable and that will also have great implications on the theory of computation.

\section{Conclusion}
\label{sec:con}
In this paper we presented small steps towards some challenges in finding an algorithm for Skolem problem.
The first step is to overcome the use of transcendental number theory as it does not scale well.
This goal is partially achieved as the $r^t \cos t \theta$ problem still needs to use that for some of the cases.
The absence of lower bounds makes it interesting to explore the lower bounds for $r^t \cos t \theta$ problem.
\par
Our second contribution is in the direction of continuization of computation.
The key idea is to interpolate the sequences with some well-behaving functions over reals and then thinking of the problem as a problem for these sequences.
This can be useful particularly because there is abundance of real analytic tools to find roots of functions. 
\par
We made three conjectures in this paper. 
The first conjecture is interesting as it asserts that Skolem problem is hard only for the cases when the effective eigenvalues are on unit circle but not roots of unity.
This conjecture seems plausible as all the challenges in solving Skolem problem also remain for the $\Sigma \cos t \theta_i$ problem.
\par 
The other two conjectures are about the power of continuization in general. 
The theory of closed real field with cosine function is known to be undecidable but restriction to one variable case is an interesting unexplored problem.
Our last conjecture is about weakness of one variable fragment of extensions of theory of reals.

\bibliography{main}

\end{document}